\documentclass[11pt]{article}
\usepackage{amsfonts,amssymb,amsmath,amsthm,amstext,amssymb,mathtools}
\usepackage{subcaption,float,color,xcolor,graphicx,enumitem}
\usepackage{mathrsfs,bm}
\usepackage{dsfont}  
\usepackage[authoryear,sort,round]{natbib}  
\usepackage{fullpage} 
\usepackage{hyperref}
\usepackage{centernot}
\usepackage{mathrsfs}
\usepackage[noabbrev,capitalise,nosort,nameinlink]{cleveref}

\newcommand*{\cball}[2]{B_{#2}(#1)}
\newcommand*{\defeq}{\coloneqq}
\newcommand*{\defterm}{\emph}
\newcommand*{\Naturals}{\mathbb{N}}

\newcommand*{\one}{\mathds{1}}
\newcommand*{\prob}[1]{\mathscr{P}(#1)}

\newcommand*{\quark}{\setbox0\hbox{$x$}\hbox to\wd0{\hss$\cdot$\hss}}
\newcommand*{\rd}{\mathrm{d}}
\newcommand*{\Reals}{\mathbb{R}}

\DeclareMathOperator*{\supp}{supp}

\renewcommand*{\emptyset}{\varnothing}

\renewcommand*{\geq}{\geqslant}

\renewcommand*{\leq}{\leqslant}

\newcommand*{\crcdf}[2]{\mu(\cball{#1}{#2})}
\newcommand*{\crcdfm}[3]{#1(\cball{#2}{#3})}

\theoremstyle{plain}
\newtheorem{theorem}{Theorem}[section]
\newtheorem{proposition}[theorem]{Proposition}
\newtheorem{lemma}[theorem]{Lemma}
\newtheorem{corollary}[theorem]{Corollary}
\theoremstyle{definition}
\newtheorem{definition}[theorem]{Definition}

\newcommand{\absval}[1]{\lvert #1 \rvert}

\newcommand{\innerprod}[2]{\langle #1 , #2 \rangle}
\newcommand{\norm}[1]{\lVert #1 \rVert}

\newcommand{\Norm}[1]{\left\Vert #1 \right\Vert}
\newcommand{\Set}[2]{\left\{ #1 \,\middle\vert\, #2 \right\}}

\numberwithin{equation}{section}
\numberwithin{figure}{section}
\numberwithin{table}{section}

\usepackage[a4paper, top=88pt, bottom=88pt, left=66pt, right=66pt, headsep=16pt, footskip=28pt, head=13.6pt]{geometry}
\usepackage{hyperref}
\hypersetup{
	colorlinks,
	linkcolor=blue,
	citecolor=blue,
	urlcolor=blue,
}
\newcommand*{\arXiv}[1]{\bgroup\color{blue}\href{https://arxiv.org/abs/#1}{arXiv:#1}\egroup}
\newcommand*{\doi}[1]{\bgroup\color{blue}\href{https://doi.org/#1}{doi:#1}\egroup}
\newcommand*{\email}[1]{\bgroup\color{blue}\href{mailto:#1}{#1}\egroup}
\renewcommand*{\url}[1]{\bgroup\color{blue}\href{#1}{#1}\egroup}
\usepackage{enumitem, moreenum}
\setlist[enumerate]{nosep}
\setlist[itemize]{nosep}
\usepackage{mleftright} \mleftright

\usepackage[labelfont={sf,bf}]{caption}
\let\oldtitle\title
\renewcommand{\title}[1]{\oldtitle{#1}\newcommand{\theshorttitle}{#1}}
\newcommand{\shorttitle}[1]{\renewcommand{\theshorttitle}{#1}}
\let\oldauthor\author
\renewcommand{\author}[1]{\oldauthor{#1}\newcommand{\theshortauthor}{#1}}
\newcommand{\theabstract}[1]{\par\bgroup\noindent\textbf{\textsf{Abstract.}} #1\egroup}
\newcommand{\thekeywords}[1]{\par\smallskip\bgroup\noindent\textbf{\textsf{Keywords.}}\newcommand{\and}{ $\bullet$ } #1\egroup}
\newcommand{\themsc}[1]{\par\smallskip\bgroup\noindent\textbf{\textsf{2020 Mathematics Subject Classification.}}\newcommand{\and}{ $\bullet$ } #1\egroup}
\newcommand{\theshortmsc}[1]{\par\smallskip\bgroup\noindent\textbf{\textsf{2020 MSC.}}\newcommand{\and}{ $\bullet$ } #1\egroup}

\newcommand*{\affilref}[1]{\ref{affiliation#1}}
\newcommand*{\affiliation}[3]{
	\footnotetext[#1]{\label{affiliation#2}#3}
}

\title{Strong maximum a posteriori estimation in Banach spaces with Gaussian priors%
}
\shorttitle{Strong MAP estimation in Banach spaces with Gaussian priors%
}
\author{
	Hefin Lambley\textsuperscript{\affilref{Warwick}}%
}

\date{\today}

\begin{document}

\maketitle
\affiliation{1}{Warwick}{Mathematics Institute, University of Warwick, Coventry, CV4 7AL, United Kingdom\newline (\email{hefin.lambley@warwick.ac.uk})
}

\begin{abstract}\small
	\theabstract{This article shows that a large class of posterior measures that are absolutely continuous with respect to a Gaussian prior have strong maximum a posteriori estimators in the sense of Dashti et al.\ (\emph{Inverse Probl.}\ \textbf{29}:095017, 2013). 
This result holds in any separable Banach space and applies in particular to nonparametric Bayesian inverse problems with additive noise. 
When applied to Bayesian inverse problems, this significantly extends existing results on maximum a posteriori estimators by relaxing the conditions on the log-likelihood and on the space in which the inverse problem is set.
}
	\thekeywords{Bayesian nonparametrics%
\and
Bayesian inverse problems%
\and%
maximum a posteriori estimation%
\and%
modes of probability measures%
}
	\theshortmsc{28C20
\and%
60B11
\and%
62F10
\and%
62F15
\and%
62G05
}
\end{abstract}

\section{Introduction}
\label{sec:introduction}

Nonparametric Bayesian models --- which have infinite-dimensional parameters such as functions --- are increasingly popular in modern statistical practice.
For inverse problems, the need for prior information to overcome ill-posedness motivates the use of a Bayesian approach, and the desire for algorithms consistent at every resolution makes the nonparametric approach \citep[as advocated by][]{Stuart2010} very appealing.
One challenge in nonparametric Bayesian inference is that the posterior is a probability distribution on an infinite-dimensional space, making it difficult to analyse and interpret.

This article studies \emph{maximum a posteriori (MAP) estimation} in nonparametric Bayesian inverse problems.
A MAP estimator is a \emph{mode} of the posterior:
a summary by a ``most likely'' point under the measure.
The usual definition of a MAP estimator as a maximiser of the Lebesgue density is not available when the posterior is a measure on an infinite-dimensional parameter space, so it is common to define modes as the centres of metric balls with asymptotically maximal probability as proposed by \citet{DashtiLawStuartVoss2013}.

This definition allows modes of probability measures to be studied in very general settings, but we restrict attention to posterior measures arising in nonparametric Bayesian inverse problems.
In particular, we study the case that $\mu^y$ is a posterior measure on a separable Banach space $X$ which is absolutely continuous with respect to a Gaussian prior $\mu_0$ and has Radon--Nikodym derivative
\begin{equation}
	\label{eq:Bayesian_posterior}
	\mu^y(\rd x) = \exp\bigl(-\Phi(x)\bigr) \,\mu_0(\rd x).
\end{equation}
The \emph{potential} $\Phi \colon X \to \Reals$ is determined by the structure of the problem of interest and is essentially the negative log-likelihood of the statistical model.
Any map $\Phi$ satisfying mild regularity conditions (\Cref{thm:Phi_normalisable}) yields a well-defined probability measure $\mu^{y}$.

A classical example giving rise to such a posterior is the nonlinear inverse problem of inferring a parameter $x \in X$, which is typically a function, from a noisy observation $y \in Y$ given a Gaussian prior $\mu_0$ for $x$, with
\begin{equation}
	\label{eq:additive-noise_BIP}
	y = \mathcal{G}(x) + \xi.
\end{equation}
The \defterm{observation operator} $\mathcal{G} \colon X \to Y$ is a measurable map relating the unknown $x \in X$ with the idealised observation $\mathcal{G}(x) \in Y$, which is often assumed to have finite dimension.
This observation is then corrupted by the additive random noise $\xi$ taking values in $Y$.
Under appropriate regularity conditions on $\mathcal{G}$ and $\xi$, the posterior $\mu^y$ for the conditional distribution $x \mid y$ is given by Bayes' rule and has the form \eqref{eq:Bayesian_posterior} for some potential $\Phi(\quark; y) \colon X \to \Reals$ (\Cref{thm:Bayes}).

In the Bayesian inverse problems literature, \citet{DashtiLawStuartVoss2013} developed the notion of a \defterm{strong mode} (\Cref{def:mode}) to define MAP estimators in the nonparametric setting and proved that strong modes exist when $X$ is a separable Hilbert space under mild assumptions on the potential $\Phi$.
They also showed that strong modes coincide with minimisers of an Onsager--Machlup (OM) functional (\Cref{def:OM_functional}) for the posterior in this setting, connecting their approach with previous work on most-likely paths of diffusion processes. 
For Bayesian inverse problems with additive noise as in \eqref{eq:additive-noise_BIP}, the OM functional can be viewed as a Tikhonov-regularised misfit functional, so the variational solution to an inverse problem --- that is, the minimiser of the Tikhonov functional --- can be viewed as a MAP estimator for a fully Bayesian approach.
This connection is a significant driver for the development of the nonparametric mode theory described here.

As pointed out by \citet{KlebanovWacker2022}, although \citet{DashtiLawStuartVoss2013} stated their results in the Banach setting, technical complications limit their proof strategy to Hilbert spaces, and some additional results are needed to complete the proof even in the Hilbert case, as described by \citet{Kretschmann2019}.

Recent work by Klebanov and Wacker extended the existence result to the sequence spaces $X = \ell^p(\Naturals; \Reals)$, $1 \leq p < \infty$, for Gaussian priors with \emph{diagonal} covariance structure with respect to the canonical basis, i.e.\ $\mu_0 = \bigotimes_{n \in \Naturals} N(0, \sigma_n^{2})$ for some $(\sigma_n)_{n \in \Naturals} \in \ell^{p}(\Naturals; \Reals)$.

This article proves the existence of strong modes for posteriors of the form \eqref{eq:Bayesian_posterior} defined on a separable Banach space with any Gaussian prior, as originally claimed by \citet[Theorem~3.5]{DashtiLawStuartVoss2013}, and shows that strong modes are equivalent in this setting to the other types of small-ball modes present in the literature: the weak mode (\Cref{def:mode}) and the generalised strong mode (\Cref{def:generalised_mode}).

\begin{theorem}
	\label{thm:main}
	Let $X$ be a separable Banach space equipped with a centred nondegenerate Gaussian prior $\mu_0$.
	Let $\mu^y$ be the corresponding Bayesian posterior of the form \eqref{eq:Bayesian_posterior} for some continuous potential $\Phi \colon X \to \Reals$, and suppose that, for each $\eta > 0$, there exists $K(\eta) \in \Reals$ such that
	\begin{equation} \label{eq:Phi_lower_bound}
		\Phi(x) \geq K(\eta) - \eta\|x\|_X^{2} \text{~~for all $x \in X$.}
	\end{equation}
	Then:
	\begin{enumerate}[label=(\alph*)]
		\item $\mu^y$ has a strong mode, i.e.\ a strong MAP estimator, and any strong mode lies in the Cameron--Martin space of $\mu_0$;
		\item strong modes, generalised strong modes, weak modes and minimisers of an OM functional for $\mu^y$ coincide.
	\end{enumerate}
\end{theorem}

The conditions imposed on the potential are weaker than those used by \citet{DashtiLawStuartVoss2013} and \citet{KlebanovWacker2022}, who assumed that the potential was globally bounded below and locally Lipschitz.
As pointed out by \citet{Kretschmann2019,Kretschmann2022}, a global lower bound on $\Phi$ excludes Bayesian inverse problems with observations corrupted by additive white noise or Laplacian noise, and the proof of \citet{DashtiLawStuartVoss2013} can be extended to handle these cases using a less restrictive lower bound.
These cases can also be treated under the yet weaker conditions used here, which are similar to the assumptions used by \citet[Assumption~2.6]{Stuart2010} in developing a well-posedness theory for nonparametric Bayesian inverse problems.

\subsection{Outline} 
\Cref{sec:notation} defines the small-ball modes used in this paper in the general setting of a metric space and states some essential results, including the strong--weak dichotomy (\Cref{lem:all_weak_or_strong}) which appears to be new to the literature.
This section also recalls properties of Gaussian measures used throughout the article and briefly outlines the motivating application of Bayesian inverse problems.

\Cref{sec:small-ball_Gaussian_Banach} states the main estimate (\Cref{prop:CM_bound}) needed to prove \Cref{thm:main}, which can be viewed as an analogue of the explicit Anderson inequality of \citet[Lemma~3.6]{DashtiLawStuartVoss2013}.
This is then used to establish the $M$-property for Gaussian measures on a separable Banach space (\Cref{cor:Gaussian_measure_M-prop}), which was until now known rigorously only for special cases such as separable Hilbert spaces and $\ell^{p}$ spaces equipped with diagonal Gaussian measures.

\Cref{sec:map_estimators} uses the tools developed in the previous section to study MAP estimators for Bayesian posteriors of the form \eqref{eq:Bayesian_posterior}.
First, it states a short proof for the existence of weak modes using the $M$-property.
Then, the bound in \cref{prop:CM_bound} is used to show that any \defterm{asymptotic maximising family} (\Cref{def:AMF}) for the posterior has a limit point (\Cref{lem:AMF_has_limit_point}).
\Cref{lem:AMF_limit_point_is_strong} shows that such a point must be a strong mode, extending a previous proof of \citet{KlebanovWacker2022} to the Banach case, and this completes the proof of \cref{thm:main}.

\Cref{sec:consistency} studies consistency theory for MAP estimators of Bayesian inverse problems of the type \eqref{eq:additive-noise_BIP}.
Using \cref{thm:main}, the consistency results of \citet{DashtiLawStuartVoss2013} are extended to apply in any separable Banach space $X$ (\cref{thm:consistency_infinite-data_limit}, \cref{thm:consistency_small-noise_limit}).

\Cref{sec:conclusion} gives some concluding remarks and suggests directions for future research.

\section{Preliminaries and related work}
\label{sec:notation}

For most of the paper, $X$ will be a separable real Banach space, although some definitions and preliminary results in this section will be given in the more general case that $X$ is a metric space.
In any metric space, the closed ball of radius $r$ will be denoted $\cball{x}{r}$.
We consider only Borel measures and denote the set of Borel probability measures on $X$ by $\prob{X}$.
When $X$ is separable, the topological support
\begin{equation*}
	\supp(\mu) \defeq \Set{x \in X}{\crcdf{x}{r} > 0 \text{ for all } r > 0}
\end{equation*}
is nonempty \citep{AliprantisBorder2006};
this ensures that the quantity $M_{r}$ defined in \eqref{eq:M_r} is strictly positive for all $r > 0$.

\subsection{Mode theory}
\label{subsection:modes}

As mentioned in the introduction, the small-ball mode theory has been developed largely in the Bayesian inverse problems literature.
Strong modes were proposed by \citet{DashtiLawStuartVoss2013}, and weak modes were later suggested by \citet{HelinBurger2015} as a more convenient definition when connecting MAP estimators with variational solutions to inverse problems.
Following \citet{AyanbayevKlebanovLieSullivan2022a}, we consider only \defterm{global} weak modes in this article.

\begin{definition}
	\label{def:mode}
	Let $X$ be a metric space and let $\mu \in \prob{X}$.
	A \defterm{weak mode} of $\mu$ is any point $x^{\star} \in \supp(\mu)$ such that, for all $x \in X$,
	\begin{equation} \label{eq:weak_mode}
		\limsup_{r \to 0} \frac{\crcdf{x}{r}}{\crcdf{x^{\star}}{r}} \leq 1.
	\end{equation}
	Suppose also that $X$ is separable. 
	Then a \defterm{strong mode} of $\mu$ is any point $x^{\star} \in \supp(\mu)$ such that
	\begin{equation} \label{eq:M_r}
		\lim_{r \to 0} \frac{\crcdf{x^{\star}}{r}}{M_r} = 1,~~~M_r \defeq \sup_{x \in X} \crcdf{x}{r}.
	\end{equation}
\end{definition}

The modes of a posterior measure $\mu^y$ will also be called \emph{MAP estimators}.
The difference between the two definitions \eqref{eq:weak_mode} and \eqref{eq:M_r} amounts to the order in which the supremum is taken:
a weak mode must have asymptotically greater mass when compared to every other point individually, whereas a strong mode must asymptotically have the supremal ball mass.
All strong modes are weak modes, because if $x^{\star}$ is a strong mode and $x \in X$, then
\begin{equation*}
	\limsup_{r \to 0} \frac{\crcdf{x}{r}}{\crcdf{x^{\star}}{r}} \leq \lim_{r \to 0} \frac{M_r}{\crcdf{x^{\star}}{r}} = 1.
\end{equation*}
\citet{LieSullivan2018} proved that the converse may be false:
there exist measures which have only weak modes and no strong modes.
While the literature on modes largely treats ``strong'' or ``weak'' as a property of the mode itself, one should really think of ``strong'' or ``weak'' as a global regularity condition on the measure, because either all modes of a measure are strong or none of them are strong, as the following result shows.

\begin{lemma}[Strong--weak dichotomy for modes]
	\label{lem:all_weak_or_strong}
	Let $X$ be a separable metric space.
	If $\mu \in \prob{X}$ has a strong mode, then all weak modes of $\mu$ are strong modes.
\end{lemma}

\begin{proof}
	Suppose that $x^{\star}$ is a strong mode and $y^{\star}$ is a weak mode.
	As both $x^{\star}$ and $y^{\star}$ are weak modes, the definitions imply that
	\begin{equation*}
		1 \leq \left(\limsup_{r \to 0} \frac{\crcdf{y^{\star}}{r}}{\crcdf{x^{\star}}{r}}\right)^{-1} = \liminf_{r \to 0} \frac{\crcdf{x^{\star}}{r}}{\crcdf{y^{\star}}{r}} \leq \limsup_{r \to 0} \frac{\crcdf{x^{\star}}{r}}{\crcdf{y^{\star}}{r}} \leq 1.
	\end{equation*}
	An application of the product rule for limits shows that $y^{\star}$ must also be a strong mode:
	\begin{equation*}
		\lim_{r \to 0} \frac{\crcdf{y^{\star}}{r}}{M_r} = \lim_{r \to 0} \frac{\crcdf{x^{\star}}{r}}{M_r} \lim_{r \to 0} \frac{\crcdf{y^{\star}}{r}}{\crcdf{x^{\star}}{r}} = 1.
		\qedhere
	\end{equation*}
\end{proof}

\citet{ClasonHelinKretschmannPiiroinen2019} proposed the \defterm{generalised strong mode}, motivated by inverse problems with hard parameter constraints (in the spirit of Ivanov regularisation) which lead to a posterior assigning zero mass outside of some feasible set.

\begin{definition}
	\label{def:generalised_mode}
	Let $X$ be a separable metric space.
	A \defterm{generalised strong mode} of $\mu \in \prob{X}$ is any point $x^{\star} \in X$ such that, for each sequence $(r_n)_{n \in \Naturals} \to 0$, there exists $(x_n)_{n \in \Naturals} \to x^{\star}$ with
	\begin{equation*}
		\lim_{n \to \infty} \frac{\crcdf{x_n}{r_n}}{M_{r_n}} = 1.
	\end{equation*}
\end{definition}

Taking the constant sequence $x_n = x^{\star}$ in the definition shows that a strong mode $x^{\star}$ is also a generalised strong mode.
Unlike strong and weak modes, generalised strong modes need not lie in the support of the measure.
Furthermore, there is no strong--generalised strong dichotomy or weak--generalised strong dichotomy analogous to \cref{lem:all_weak_or_strong}:
for the measure on $\Reals$ with Lebesgue density $\rho(x) = \one\{x \in [0, 1]\}$, any $x \in (0, 1)$ is a strong mode (and hence a weak mode), but the points $x = 0$ and $x = 1$ are only generalised strong modes and are neither strong modes nor weak modes.

An alternative approach to find ``most likely'' points is to minimise an OM functional associated with the measure of interest. This arises from the study of most-probable paths of diffusion processes \citep{DuerrBach1978}.

\begin{definition}
	\label{def:OM_functional}
	Let $X$ be a metric space and let $\mu \in \prob{X}$.
	Suppose that $\emptyset \neq E \subseteq \supp(\mu)$.
	A function $I \colon E \to \Reals$ is called an \defterm{Onsager--Machlup functional} for $\mu$ if, for all $x, x' \in E$,
	\begin{equation*}
		\lim_{r \to 0} \frac{\crcdf{x}{r}}{\crcdf{x'}{r}} = \exp\left(I(x') - I(x)\right).
	\end{equation*}
\end{definition}

OM functionals are unique up to additive constants and can be interpreted heuristically as the negative logarithm of the Lebesgue density --- but this cannot be taken literally for measures on an infinite-dimensional space, where there is no Lebesgue measure.
For example, an OM functional for a Gaussian measure on an infinite-dimensional Banach space can be defined only on a small subspace called the \emph{Cameron--Martin space} (see \eqref{eq:Gaussian_OM}).
As an OM functional need not be defined on the entire space $X$, it is not immediate that an OM minimiser is in any sense ``most likely'' under the measure $\mu$, and this is the motivation to study small-ball modes as in \cref{def:mode} instead.
A weak mode is always a minimiser of any OM functional for $\mu$, however, and the \emph{$M$-property} of \citet{AyanbayevKlebanovLieSullivan2022a} gives a sufficient condition to ensure that an OM minimiser is a weak mode.

\begin{definition}[$M$-property]
	Let $X$ be a metric space and let $\mu \in \prob{X}$.
	\emph{Property $M(\mu, E)$} holds for the set $\emptyset \neq E \subseteq \supp(\mu)$ if there exists $x^{\star} \in E$ such that
	\begin{equation*}
		\lim_{r \to 0} \frac{\crcdf{x}{r}}{\crcdf{x^{\star}}{r}} = 0 \text{~~for all $x \notin E$.}
	\end{equation*}
\end{definition}

The next result states this equivalence between OM minimisers and weak modes under the $M$-property and shows that the $M$-property is inherited by a posterior of the form \eqref{eq:Bayesian_posterior} from the prior. 
This generalises Proposition~4.1 and Lemma~B.8 of \citet{AyanbayevKlebanovLieSullivan2022a} to potentials that are merely continuous rather than locally uniformly continuous.
In the specific case that $\mu_{0}$ is a Gaussian measure on a separable Banach space $X$, the claim \ref{item:OM_minimisers_are_modes_1} generalises Theorem~3.2 of \citet{DashtiLawStuartVoss2013}, which requires that the potential is locally bounded and Lipschitz.

\begin{proposition}
	\label{prop:OM_minimisers_are_modes}
	Let $X$ be a metric space and suppose that $\mu_{0} \in \prob{X}$ has OM functional $I_{0} \colon E \to \Reals$.
	Suppose that property $M(\mu_0, E)$ holds and that $\mu^{y}$ is a probability measure on $X$ of the form \eqref{eq:Bayesian_posterior} for some continuous potential $\Phi \colon X \to \Reals$.
	Then:
	\begin{enumerate}[label=(\alph*)]
		\item
		\label{item:OM_minimisers_are_modes_1}
		$\mu^y$ has OM functional $I^y \colon E \to \Reals$ given by $I^y(u) = I_0(u) + \Phi(u)$ and property $M(\mu^y, E)$ holds;

		\item
		\label{item:OM_minimisers_are_modes_2}
		$x^{\star} \in X$ is a weak mode for $\mu^{y}$, i.e.\ a weak MAP estimator, if and only if $x^{\star} \in E$ and $x^{\star}$ minimises $I^y$.
	\end{enumerate}
\end{proposition}

\begin{proof}
	Let $x \in X$ and $x' \in \supp(\mu_{0})$.
	As the density $\exp(-\Phi)$ is strictly positive, $x' \in \supp(\mu^{y})$ and thus $\mu^{y}(\cball{x'}{r}) > 0$ for all $r > 0$.
	By the continuity of $\Phi$, for each $\varepsilon > 0$ there exists $\delta > 0$ such that $\norm{u - x}_X < \delta \implies \absval{\Phi(u) - \Phi(x)} < \varepsilon$ and $\norm{u - x'}_X < \delta \implies \absval{\Phi(u) - \Phi(x')} < \varepsilon$.
	Hence, for $r < \delta$, it follows that
	\begin{align}
		\notag
		\frac{\mu^y(\cball{x}{r})}{\mu^y(\cball{x'}{r})} &= \frac{\exp(-\Phi(x)) \int_{\cball{x}{r}} \exp(\Phi(x) - \Phi(u)) \,\mu_0(\rd u)}{\exp(-\Phi(x')) \int_{\cball{x'}{r}} \exp(\Phi(x') - \Phi(u)) \,\mu_0(\rd u)} \\
		\label{eq:posterior_ratio}
		&< \exp\Bigl(\Phi(x') - \Phi(x) + 2\varepsilon\Bigr) \frac{\mu_0(\cball{x}{r})}{\mu_0(\cball{x'}{r})}.
	\end{align}
	Property $M(\mu^y, E)$ follows immediately from \eqref{eq:posterior_ratio} by choosing $x \notin E$, $x' \in E$ and taking the $\limsup$ as $r \to 0$.
	To obtain the OM functional $I^{y}$, suppose instead that $x, x' \in E$; then by \eqref{eq:posterior_ratio} and using the OM functional $I_{0}$ for $\mu_{0}$,
	\begin{align*}
		\limsup_{r \to 0} \frac{\mu^{y}(\cball{x}{r})}{\mu^{y}(\cball{x'}{r})}
		&\leq \exp\Bigl( \Phi(x') - \Phi(x) + 2\varepsilon \Bigr) \limsup_{r \to 0} \frac{\mu_{0}(\cball{x}{r})}{\mu_{0}(\cball{x'}{r})} \\
		&= \exp\Bigl(\Phi(x') - \Phi(x) + 2\varepsilon + I_{0}(x') - I_{0}(x) \Bigr).
	\end{align*}
	By deriving a lower bound analogous to \eqref{eq:posterior_ratio} using the continuity of $\Phi$ and taking the $\liminf$ as $r \to 0$, we obtain the inequality
	\begin{align*}
		\exp\Bigl(\Phi(x') - \Phi(x) + I_{0}(x') - I_{0}(x) - 2\varepsilon \Bigr) &= \exp\Bigl( \Phi(x') - \Phi(x) - 2\varepsilon \Bigr) \liminf_{r \to 0} \frac{\mu_{0}(\cball{x}{r})}{\mu_{0}(\cball{x'}{r})} \\
		&\leq \liminf_{r \to 0} \frac{\mu^{y}(\cball{x}{r})}{\mu^{y}(\cball{x'}{r})}.
	\end{align*}
	As $\varepsilon > 0$ is arbitrary this proves that $I^{y}(u) = I_{0}(u) + \Phi(u)$.
	
	The claim in \ref{item:OM_minimisers_are_modes_2} is an immediate consequence of \citet[Proposition~4.1]{AyanbayevKlebanovLieSullivan2022a}.
\end{proof}

Thus, when property $M(\mu^y, E)$ holds, one can view $I^y$ as an extended-real-valued function with value $+\infty$ outside $E$.
This interpretation is not valid if the $M$-property does not hold and one can say very little about the behaviour of $\mu^y$ on balls centred outside of $E$ using an OM functional in this case.

While this article considers only Gaussian priors, MAP estimators have also been studied for Bayesian inverse problems with Besov and Cauchy priors \citep{AgapiouBurgerDashtiHelin2018,AyanbayevKlebanovLieSullivan2022b}.
Besov and Cauchy priors are typically constructed as product measures placing full mass on a Banach subspace of $\Reals^\infty$, and the product structure of $\Reals^\infty$ makes finite-dimensional approximation arguments possible.
As an arbitrary Banach space need not have such product structure, we instead exploit the fact that a Gaussian measure is fully determined by its behaviour on a Hilbert subspace (the Cameron--Martin space) whose geometry is much more convenient to work with.

\subsection{Gaussian measures}
This section summarises the properties of Gaussian measures used in the article;
see the monograph of \citet{Bogachev1998} for a thorough introduction to Gaussian measures.
If $X$ is a separable Banach space, a measure $\gamma \in \prob{X}$ is \emph{Gaussian} if the pushforward $\gamma \circ f^{-1}$ is a Gaussian measure on $\Reals$ for every $f$ lying in the topological dual $X^{\ast}$.
The measure $\gamma$ is centred if it has mean zero and nondegenerate if it has full support, i.e.\ $\supp(\gamma) = X$;
we assume that $\gamma$ is always centred and nondegenerate in the remainder of the article.

The \defterm{reproducing-kernel Hilbert space} (RKHS) $X_{\gamma}^{\ast}$ of $\gamma$ is the $L^{2}(\gamma)$-closure of $X^{\ast}$, and the covariance operator $R_\gamma \colon X_\gamma^{\ast} \to (X^{\ast})'$, taking values in the algebraic dual $(X^{\ast})'$ of $X^{\ast}$, is given by
\begin{equation*}
	R_\gamma(f)(g) \defeq \int_X f(x) g(x) \,\gamma(\rd x).
\end{equation*}
As $X$ is separable, the measure $\gamma$ is Radon and thus $R_\gamma(f)$ is representable by an element of $X$ for any $f \in X^{\ast}_{\gamma}$ \citep[Theorem~3.2.3]{Bogachev1998}.
The image of $R_\gamma$ in $X$ is called the \defterm{Cameron--Martin space} $E \subset X$.
It is a separable Hilbert space under the Cameron--Martin inner product $\innerprod{h}{k}_E \defeq \innerprod{R_\gamma^{-1} h}{R_\gamma^{-1} k}_{L^2(\gamma)}$, which induces the norm $\norm{h}_{E} \defeq \norm{R_\gamma^{-1} h}_{L^{2}(\gamma)}$.
The Cameron--Martin space of a Radon Gaussian measure $\gamma$ is compactly embedded in $X$, i.e.\ there exists $C > 0$ such that 
\begin{equation} \label{eq:compact_embedding}
	\norm{x}_X \leq C \norm{x}_{E}
\end{equation}
and the inclusion $\iota \colon E \to X$ is a compact operator \citep[Corollary~3.2.4]{Bogachev1998}.
In particular, any $E$-weakly convergent sequence is mapped by $\iota$ to an $X$-strongly convergent sequence.

The covariance operator $R_\gamma \colon X_\gamma^{\ast} \to E$ is a Hilbert isometric isomorphism between the RKHS (equipped with the $L^2(\gamma)$-inner product) and the Cameron--Martin space (equipped with the Cameron--Martin inner product).

The Cameron--Martin space for $\gamma$ is precisely the set of all directions $h \in X$ for which the shifted measure $\gamma_h(\quark) \defeq \gamma(\quark - h)$ is absolutely continuous with respect to $\gamma$.
The space $E$ has $\gamma$-measure zero, but if $\gamma$ is nondegenerate then $E$ is dense in $X$.
When $h \in E$, the density of the shifted measure $\gamma_{h}$ with respect to $\gamma$ is given by the \defterm{Cameron--Martin formula} \citep[Corollary~2.4.3]{Bogachev1998},
\begin{equation} \label{eq:Cameron--Martin_formula}
	\gamma_h(\rd x) = \exp\left( \bigl(R_\gamma^{-1}h\bigr)(x) - \frac{1}{2}\norm{h}_{E}^{2} \right) \,\gamma(\rd x).
\end{equation}
If $h \notin E$, then the measures $\gamma$ and $\gamma_h$ are mutually singular by the Feldman--H\'ajek theorem \citep[Theorem~2.7.2]{Bogachev1998}.

A centred Gaussian measure $\gamma$ has OM functional
\begin{equation} \label{eq:Gaussian_OM}
	I \colon E \to \Reals,~~I(u) = \frac{1}{2} \norm{u}_{E}^{2},
\end{equation}
which is defined only on the Cameron--Martin space $E$.
Property $M(\gamma, E)$ is known to hold when $X$ is a separable Hilbert space, as proven by \citet[Corollary~3.8]{DashtiLawStuartVoss2013} and \citet[Corollary~5.2]{AyanbayevKlebanovLieSullivan2022a}, and when $X = \ell^p$, $1 \leq p < \infty$, provided that $\gamma$ has diagonal covariance structure \citep[Lemma~4.5]{KlebanovWacker2022}.
The measure $\gamma$ also satisfies \defterm{Anderson's inequality} \citep[Theorem~2.8.10]{Bogachev1998}:
\begin{equation} \label{eq:Anderson_inequality}
	\crcdfm{\gamma}{x}{r} \leq \crcdfm{\gamma}{0}{r} \text{~~for any $x \in X$ and $r > 0$.}
\end{equation}
Gaussian measures do not charge the boundaries of metric balls, i.e.\ $\gamma(\partial \cball{x}{r}) = 0$ \citep[see e.g.][Lemma~6.1]{AgapiouBurgerDashtiHelin2018}, so it would be equivalent to use open balls in any of the results in this article.

The tail behaviour of a Gaussian measure is described by Fernique's theorem \citep{Fernique1970}, and this is the chief reason for the lower bound \eqref{eq:Phi_lower_bound} on the potential needed in \cref{thm:main}. 
Fernique's theorem states that for any Gaussian measure $\gamma$ on a separable Banach space $X$, there exists $\eta > 0$ such that
\begin{equation*}
	\int_X \exp\left(\eta \norm{x}_X^2\right) \,\gamma(\rd x) < \infty.
\end{equation*}

In the rest of the article, $\gamma$ will denote a centred nondegenerate Gaussian measure and the prior measure $\mu_0$ will always be a centred nondegenerate Gaussian;
in either case, $E$ will denote the corresponding Cameron--Martin space.

\subsection{Bayesian inverse problems}
\label{sec:BIPs}

Ill-posed inverse problems are challenging to solve and require the use of prior information about the solution $x$ to restore the well-posedness of the problem.
The motivating example in this article is the nonlinear inverse problem of recovering an infinite-dimensional parameter (e.g.\ a function) $x \in X$ from a noisy observation of the finite-dimensional quantity $y = \mathcal{G}(x)$, as discussed in the introduction.

Well-posedness is essential to allow for numerical solution of inverse problems, and the classical approach to restoring well-posedness uses regularisation \citep[see e.g.][]{BenningBurger2018}:
a variational solution to the inverse problem \eqref{eq:additive-noise_BIP} is a minimiser of the \emph{Tikhonov functional}
\begin{equation*}
	F(x) = \Norm{\mathcal{G}(x) - y} + \alpha \Norm{x}',
\end{equation*}
where $\Norm{\quark}'$ is some norm penalising undesirable properties of the solution $x$, e.g.\ the total-variation norm of the function $x$ \citep{RudinOsherFatemi1992}.

In contrast, the Bayesian approach incorporates prior information using a prior measure on the solution space.
As stated in the next theorem, under mild conditions on the prior $\mu_{0}$ and on the potential $\Phi$ arising from the observation operator $\mathcal{G}$, an analogue of Bayes' rule gives an expression for the posterior for $x \mid y$ on the infinite-dimensional parameter space.

\begin{theorem}[{\citealp[Theorem~14]{DashtiStuart2017}}]
	\label{thm:Bayes}
	Let $X$ and $Y$ be separable Banach spaces and suppose that $\mathcal{G} \colon X \to Y$ is measurable.
	Suppose that $x$ has prior distribution $\mu_{0} \in \prob{X}$ and
	\begin{equation*}
		y = \mathcal{G}(x) + \xi,
	\end{equation*}
	where $\xi$ is random noise with distribution $\tau_{0} \in \prob{Y}$, which is assumed to be independent of $x$.
	Suppose that the translated measure $\tau_{\mathcal{G}(x)}(\quark) \defeq \tau_{0}(\quark - \mathcal{G}(x))$ is absolutely continuous with respect to $\tau_{0}$ for $\mu_{0}$-almost all $x \in X$ and define the potential
	\begin{equation*}
		\Phi(x; y) \defeq -\log \frac{\rd \tau_{\mathcal{G}(x)}}{\rd \tau_{0}}(y).
	\end{equation*}
	Suppose further that $\Phi \colon X \times Y \to \Reals$ is measurable with respect to the product measure $\mu_{0} \otimes \tau_{0}$, and that for $\tau_{0}$-almost all $y \in Y$,
	\begin{equation} \label{eq:normalisation_constant}
		Z(y) \defeq \int_{X} \exp\Bigl(-\Phi(x; y)\Bigr) \,\mu_{0}(\rd x) > 0.
	\end{equation}
	Then the conditional distribution $\mu^{y}$ of $x \mid y$ exists, is absolutely continuous with respect to $\mu_{0}$, and
	\begin{equation} \label{eq:posterior_density_BIP}
		\frac{\rd \mu^{y}}{\rd \mu_{0}}(x) = \frac{1}{Z(y)} \exp\Bigl( -\Phi(x; y) \Bigr).
	\end{equation}
\end{theorem}

The data $y \in Y$ will be considered fixed and we suppress the explicit dependence on $y$; thus, the potential is a map $\Phi \colon X \to \Reals$.
When $y$ has finite dimension and $\tau_{0}$ is absolutely continuous with respect to the Lebesgue measure, $\tau_{\mathcal{G}(x)}$ is absolutely continuous with respect to $\tau_{0}$ and $\Phi(x)$ can typically be interpreted as a misfit functional: 
when $\xi$ has mean-zero Gaussian distribution $\xi \sim N(0, \Sigma)$ on $Y = \Reals^{d}$, for example, one can take
\begin{equation}
	\label{eq:Gaussian-noise_BIP_potential}
	\Phi(x) \propto \norm{\mathcal{G}(x) - y}_{\Sigma}^{2},~~~~~~\norm{\quark}_{\Sigma} = \norm{\Sigma^{-1/2} \quark}.
\end{equation}
By absorbing the normalisation factor $\tfrac{1}{Z(y)}$ into $\Phi$, the posterior \eqref{eq:posterior_density_BIP} can be expressed in the form \eqref{eq:Bayesian_posterior} discussed in the introduction.

To ensure that the posterior measure is normalisable for a given potential $\Phi \colon X \to \Reals$, i.e.\ is a probability measure, we impose mild conditions on the form of the potential $\Phi$ in \cref{thm:main}.
If the measure $\mu^{y}$ does indeed arise from an inverse problem as in \cref{thm:Bayes}, the following result is merely a sufficient condition to ensure that $Z(y) > 0$ in \eqref{eq:normalisation_constant}.

\begin{theorem}[{\citealp[Theorem~4.1]{Stuart2010}}]
	\label{thm:Phi_normalisable}
	Suppose that the potential $\Phi \colon X \to \Reals$ is continuous and that for each $\eta > 0$, there exists a constant $K(\eta) \in \Reals$ such that
	\begin{equation} \label{eq:density_lower_bound_normalisable}
		\Phi(x) \geq K(\eta) - \eta\norm{x}_X^2 \text{~~for all $x \in X$.}
	\end{equation}
	Then the posterior measure $\mu^{y}$ given by \eqref{eq:Bayesian_posterior} can be normalised to yield a probability measure.
\end{theorem}

\begin{proof}
	Given the unnormalised density $\exp(-\Phi)$, one can normalise to obtain a probability measure with density $\exp(-\Phi')$ by setting $\Phi' = \Phi - \log Z$ with the finite normalisation constant
	\begin{equation*}
		Z \defeq \int_X \exp\Bigl(-\Phi(x)\Bigr)\,\mu_0(\rd x) \leq \int_X \exp\Bigl(-K(\eta) + \eta \norm{x}_X^2\Bigr) \,\mu_{0}(\rd x) < \infty,
	\end{equation*}
	where the upper bound follows by applying \eqref{eq:density_lower_bound_normalisable} with an appropriate $\eta > 0$ such that the integral is finite by Fernique's theorem.
\end{proof}

As discussed, a significant reason for studying MAP estimators is that they connect the Bayesian and variational approaches to inverse problems.
When the $M$-property holds, the weak MAP estimators of a Bayesian inverse problem coincide with minimisers of an OM functional, and when a Gaussian prior is used, an OM functional for the posterior has the form of a Tikhonov functional \citep[see e.g.][]{DashtiLawStuartVoss2013}. 
This correspondence depends on the $M$-property, which until now has been shown only for Gaussian measures on separable Hilbert spaces and for diagonal Gaussian measures on $X = \ell^{p}$, $1 \leq p < \infty$.
This article therefore extends the connection between Bayesian and variational approaches to Banach spaces.

\section{Small-ball probabilities for Gaussian measures in Banach spaces}
\label{sec:small-ball_Gaussian_Banach}

The main technical result required for the proof of \cref{thm:main} is the following bound on the ratio of the measures of small balls under a Gaussian measure stated in \cref{prop:CM_bound}.
This bound is similar in spirit to the explicit Anderson inequality of \citet[Lemma~3.6]{DashtiLawStuartVoss2013}, which takes the form
\begin{equation}
	\label{eq:explicit_Anderson_inequality_in_X}
	\frac{\gamma(\cball{x}{r})}{\gamma(\cball{0}{r})} \leq \exp\Bigl( a\left(r^{2} - (\norm{x}_{X} - r)^{2} \right) \Bigr)
\end{equation}
when $\gamma$ is a centred nondegenerate Gaussian measure on the separable Banach space $X$, $a = a(\gamma) > 0$, $x \in X$ and $r > 0$.
Both \eqref{eq:explicit_Anderson_inequality_in_X} and the bound we prove in \cref{prop:CM_bound} may be thought of as quantitative analogues of the Anderson inequality \eqref{eq:Anderson_inequality}.
In contrast to the inequality \eqref{eq:explicit_Anderson_inequality_in_X}, which is written in terms of the ambient norm of the Banach space, the result here is written in terms of the \emph{decentring function} \citep{GhosalVanDerVaart2017} given by
\begin{equation*}
	\psi_{x}(r) \defeq \inf_{h \in E \cap \cball{x}{r}} \frac{1}{2} \norm{h}_{E}^{2}.
\end{equation*}
We will show in \cref{prop:CM_bound} that the infimum in the decentring function is attained by some point $h^{\star} \in E$, justifying the use of a minimum instead.

When $X$ is a separable Hilbert space, the Cameron--Martin norm can be viewed as a reweighting of the norm of $X$ and the bound \eqref{eq:explicit_Anderson_inequality_in_X} in $X$-norm suffices to prove the desired results on MAP estimators.
In a Banach space, however, this is no longer true --- thus, writing the bound in terms of the Cameron--Martin norm is a natural generalisation, with the compact embedding \eqref{eq:compact_embedding} providing the means to relate the two norms.

\begin{proposition}[Explicit Anderson inequality in Cameron--Martin norm]
	\label{prop:CM_bound}
	Let $X$ be a separable Banach space equipped with a centred nondegenerate Gaussian measure $\gamma$.
	For any $x \in X$ and $r > 0$,
	\begin{equation*}
		\frac{\crcdfm{\gamma}{x}{r}}{\crcdfm{\gamma}{0}{r}} \leq \exp\left(-\min_{h \in E \cap \cball{x}{r}} \frac{1}{2} \norm{h}_{E}^{2} \right).
	\end{equation*}
\end{proposition}

\begin{proof}
	This is an immediate corollary of \citet[Proposition~11.19]{GhosalVanDerVaart2017}, and we give a version of the proof here.
	The set $E \cap \cball{x}{r}$ is nonempty (as $\gamma$ is nondegenerate), $E$-closed (as it is the preimage of $\cball{x}{r}$ under the continuous embedding $\iota \colon E \to X$) and convex.
	This implies that $E \cap \cball{x}{r}$ is $E$-weakly closed.
	Hence, the $E$-weakly lower semicontinuous map $h \mapsto \|h\|_{E}^{2}$ defined on $E \cap \cball{x}{r}$ attains its minimum on some $h^{\star} = R_\gamma g^{\star} \in E$.
	The Cameron--Martin formula \eqref{eq:Cameron--Martin_formula} gives the equality 
	\begin{equation} \label{eq:Cameron--Martin_bound}
		\crcdfm{\gamma}{x}{r} = \gamma_{-h^{\star}}(\cball{x - h^{\star}}{r}) = \exp\left(-\frac{1}{2}\norm{h^{\star}}_{E}^{2}\right) \int_{\cball{x - h^{\star}}{r}} \exp\bigl(-g^{\star}(u)\bigr) \,\gamma(\rd u),
	\end{equation}
	and we now show that $g^{\star}(u) \geq 0$ for $\gamma$-almost all $u \in \cball{x - h^{\star}}{r}$.
	As $E \cap \cball{x}{r}$ is convex and $h^{\star}$ minimises the $E$-norm on $E \cap \cball{x}{r}$, it follows that
	\begin{equation*}
		\norm{ (1 - \lambda) h^{\star} + \lambda h }_{E}^{2} \geq \norm{ h^{\star} }_{E}^{2} \text{~~for any $h \in E \cap \cball{x}{r}$ and $\lambda \in [0, 1]$.}
	\end{equation*}
	Rearranging and taking limits as $\lambda \to 0$ shows that
	\begin{equation} \label{eq:E-inner_product_bound}
		\langle h, h^{\star}\rangle_{E} \geq \|h^{\star}\|_{E}^{2} \text{~~for any $h \in E \cap \cball{x}{r}$.}
	\end{equation}
	Now, let $(h_n)_{n \in \Naturals}$ be any orthonormal basis of $E$.
	As the covariance operator $R_\gamma$ is a Hilbert isomorphism and $\sum_{k = 1}^n h_k \langle h_k, h^{\star} \rangle_{E} \to h^{\star}$ in $E$ as $n \to \infty$, it follows that
	\begin{align*}
		g_n &\defeq \sum_{k = 1}^n  (R_\gamma^{-1} h_k) \langle h_k, h^{\star} \rangle_{E} = R_\gamma^{-1} \left[ \sum_{k = 1}^n h_k \langle h_k, h^{\star} \rangle_{E} \right] \to g^{\star} \text{~~in $L^{2}(\gamma)$ as $n \to \infty$.}
	\end{align*}
	Hence, there is a subsequence $(g_{n_k})_{k \in \Naturals}$ converging pointwise $\gamma$-almost everywhere to $g^{\star}$.
	By \citet[Theorem~3.5.1]{Bogachev1998}, $\gamma$-almost all elements $u \in \cball{x - h^{\star}}{r}$ may be written as 
	\begin{equation*}
		u = \sum_{k \in \Naturals} (R_\gamma^{-1} h_k)(u) h_k,
	\end{equation*}
	where the convergence of the series is in the norm of $X$.
	Hence, for all $n$ sufficiently large and $u \in \cball{x - h^{\star}}{r}$,
	\begin{equation*}
		\sum_{k = 1}^n (R_\gamma^{-1} h_k)(u) h_k + h^{\star} \in \cball{x}{r}.
	\end{equation*}
	Using \eqref{eq:E-inner_product_bound}, we observe that
	\begin{equation*}
		g_{n}(u) + \norm{h^{\star}}_{E}^2 = \left\langle \sum_{k = 1}^n (R_{\gamma}^{-1} h_{k})h_{k} + h^{\star}, h^{\star} \right\rangle_{E} \geq \norm{h^{\star}}_{E}^{2},
	\end{equation*}
	so it immediately follows that $g_{n}(u) \geq 0$.
	As $g_{n_k}(u) \to g^{\star}(u)$ as $k \to \infty$ $\gamma$-almost everywhere, we obtain the claimed lower bound $g^{\star}(u) = \lim_{k \to \infty} g_{n_k}(u) \geq 0$ for $\gamma$-almost all $u \in \cball{x - h^{\star}}{r}$.
	The result follows by bounding the integrand in \eqref{eq:Cameron--Martin_bound} and using Anderson's inequality \eqref{eq:Anderson_inequality}:
	\begin{align*}
		\crcdfm{\gamma}{x}{r} &= \exp\left(-\frac{1}{2} \norm{h^{\star}}_{E}^{2}\right) \int_{\cball{x - h^{\star}}{r}} \exp\bigl(-g^{\star}(u)\bigr) \,\gamma(\rd u) \\
		&\leq \exp\left(-\frac{1}{2} \norm{h^{\star}}_{E}^{2}\right) \crcdfm{\gamma}{x - h^{\star}}{r} \\
		&\leq \exp\left(-\frac{1}{2} \norm{h^{\star}}_{E}^{2}\right) \crcdfm{\gamma}{0}{r}. \qedhere
	\end{align*}
\end{proof}

Though we shall not make use of \eqref{eq:explicit_Anderson_inequality_in_X}, it can be proven easily from \cref{prop:CM_bound} by applying the compact embedding \eqref{eq:compact_embedding}.

The following corollary on the measure of balls with centres converging to some $x^{\star} \in E$ is slightly weaker than the corresponding results of \citet[Lemma~4.14]{Kretschmann2019} and \citet[Lemma~A.2]{KlebanovWacker2022}, but it is sufficient for our purposes.
The proof stated here takes advantage of the bound developed in \cref{prop:CM_bound}.

\begin{corollary}
	\label{cor:Gaussian_measure_limits}
	Let $X$ be a separable Banach space equipped with a centred nondegenerate Gaussian measure $\gamma$.
	Suppose that $(x_{n}, r_{n})_{n \in \Naturals} \subset X \times [0, \infty)$ converges to $(x^{\star}, 0)$.
	After passing to a subsequence without relabelling,
	\begin{equation*}
		\limsup_{n \to \infty} \frac{\crcdfm{\gamma}{x_{n}}{r_{n}}}{\crcdfm{\gamma}{x^{\star}}{r_{n}}} \leq 1.
	\end{equation*}
\end{corollary}

\begin{proof}
	Construct the sequence $(h_n)_{n \in \Naturals} \subset E$ by selecting a minimiser (which exists as argued in the proof of \cref{prop:CM_bound}) of $h \mapsto \|h\|_{E}^{2}$ from $E \cap \cball{x_n}{r_n}$.
	Using the OM functional $I$ defined by \eqref{eq:Gaussian_OM} for $\gamma$, which satisfies $I(0) = 0$, and by applying the upper bound from \cref{prop:CM_bound}, we may write
	\begin{align*}
		\limsup_{n \to \infty} \frac{\crcdfm{\gamma}{x_n}{r_n}}{\crcdfm{\gamma}{x^{\star}}{r_n}} &= \limsup_{n \to \infty} \frac{\crcdfm{\gamma}{x_n}{r_n}}{\crcdfm{\gamma}{0}{r_n}} \lim_{n \to \infty} \frac{\crcdfm{\gamma}{0}{r_n}}{\crcdfm{\gamma}{x^{\star}}{r_n}} \\
		&\leq \limsup_{n \to \infty} \exp\left(-\frac{1}{2}\|h_n\|_{E}^{2} + I(x^{\star})\right).
	\end{align*}
	If $(h_n)_{n \in \Naturals}$ has no $E$-bounded subsequence, then the claim follows immediately as the limit on the right-hand side is zero.
	Otherwise, pass to an $E$-bounded subsequence and, by reflexivity of $E$, pass to a further $E$-weakly convergent subsequence which we do not relabel.
	Since $(x_{n})_{n \in \Naturals} \to x^{\star}$, it follows that $(h_{n})_{n \in \Naturals} \to x^{\star}$ in $X$ as $\norm{h_{n} - x_{n}}_X \leq r_{n}$, and by the compact embedding of $E$ in $X$, the $E$-weak limit of $(h_n)_{n \in \Naturals}$ must agree with the $X$-strong limit.
	Hence, $(h_n)_{n \in \Naturals} \rightharpoonup x^{\star}$ weakly in $E$, and as the Cameron--Martin norm is $E$-weakly lower semicontinuous,
	\begin{equation*}
		\limsup_{n \to \infty} \exp\left(-\frac{1}{2}\norm{h_{n}}_{E}^{2}\right) \leq \exp\left(-\frac{1}{2}\|x^{\star}\|_{E}^{2} \right) = \exp(-I(x^{\star})). \qedhere
	\end{equation*}
\end{proof}

The next result establishes a technical approximation condition for sequences in $X$ by elements of $E$, which is useful in combination with \cref{prop:CM_bound}, and applies it to establish property $M(\gamma, E)$ for Gaussian measures on Banach spaces.
As discussed in \cref{sec:notation}, this extends previous results which establish the $M$-property when $X$ is a separable Hilbert space or when $X = \ell^p$, $1 \leq p < \infty$, and $\gamma$ is a diagonal Gaussian measure.
In particular, this is a natural analogue for Banach spaces of Corollary~3.8 of \citet{DashtiLawStuartVoss2013}, which proves the $M$-property in separable Hilbert spaces.

\begin{corollary}
	\label{cor:Gaussian_measure_M-prop}
	Let $X$ be a separable Banach space equipped with a centred nondegenerate Gaussian measure $\gamma$.
	\begin{enumerate}[label=(\alph*)]
		\item
		\label{item:Gaussian_measure_seq_limit_point}
		Let $(r_n)_{n \in \Naturals} \to 0$ and $(x_{n})_{n \in \Naturals} \subset X$.
		If
		\begin{equation*}
			\liminf_{n \to \infty} \min_{h \in E \cap \cball{x_{n}}{r_n}} \frac{1}{2}\norm{h}_{E}^{2} < \infty,
		\end{equation*}
		then $(x_{n})_{n \in \Naturals}$ has an $X$-strong limit point which lies in $E$.

		\item
		\label{item:Gaussian_measure_M-prop}
		Property $M(\gamma, E)$ holds.
	\end{enumerate}
\end{corollary}

\begin{proof}
	\begin{enumerate}[label=(\alph*)]
		\item
		By hypothesis, there must exist a subsequence $(r_{n_k})_{k \in \Naturals}$ and a sequence $(h_{n_k})_{k \in \Naturals} \subset E$ that is uniformly bounded in $E$ such that $\norm{h_{n_k} - x_{n_k}}_X \leq r_{n_k}$.
		Pass to an $E$-weakly convergent subsequence of $(h_{n_k})_{k \in \Naturals}$ with limit $h^{\star} \in E$;
		the compact embedding of $E$ in $X$ implies that $(h_{n_k})_{k \in \Naturals}$ converges strongly in $X$ to $h^\star$.
		As $\norm{h_{n_k} - x_{n_k}}_X \leq r_{n_k}$, this implies that $(x_{n_k})_{k \in \Naturals} \to h^{\star}$ strongly in $X$.

		\item
		Let $x \notin E$.
		The constant sequence $(x)_{n \in \Naturals}$ cannot have a limit point in $E$, so by \ref{item:Gaussian_measure_seq_limit_point}, for any sequence $(r_n)_{n \in \Naturals} \to 0$,
		\begin{equation*}
			\lim_{n \to \infty} \min_{h \in E \cap \cball{x}{r_n}} \frac{1}{2}\|h\|_{E}^{2} = \infty.
		\end{equation*}
		Thus, by \cref{prop:CM_bound}, the $M$-property holds because
		\begin{equation*}
			\limsup_{r \to 0} \frac{\crcdfm{\gamma}{x}{r}}{\crcdfm{\gamma}{0}{r}} \leq \limsup_{r \to 0} \exp\left( -\min_{h \in E \cap \cball{x}{r}} \frac{1}{2} \norm{h}_{E}^{2} \right) = 0.
			\qedhere
		\end{equation*}
	\end{enumerate}
\end{proof}

\section{Existence of MAP estimators}
\label{sec:map_estimators}

\subsection{Weak MAP estimators}

With the $M$-property established for Gaussian measures on a separable Banach space, it is now possible to provide a short proof of the existence of weak MAP estimators for Bayesian posteriors of the form \eqref{eq:Bayesian_posterior}.
One could prove the existence of strong MAP estimators directly, as in \citet{DashtiLawStuartVoss2013}, and use the fact that all strong modes are weak modes, but it is instructive to prove the existence of weak MAP estimators separately.
Though weak modes were not proposed until the work of \citet{HelinBurger2015}, Corollary~3.8 of \citet{DashtiLawStuartVoss2013} already proved what is now called the $M$-property for Gaussian priors on Hilbert spaces, taking an important step towards showing the existence of weak modes.

By \cref{prop:OM_minimisers_are_modes}, it is sufficient to minimise the posterior OM functional $I^y$, and it is well known that $I^y$ does indeed have a minimiser \citep[see e.g.][Theorem~5.4]{Stuart2010} under mild conditions.

In particular, we only require \emph{coercivity} of $I^y$ in $E$ to obtain weak modes rather than the lower bound on $\Phi$ needed in \cref{thm:main}.
It is important to note that without the lower bound on $\Phi$, it may not be possible to normalise the measure defined in \eqref{eq:Bayesian_posterior} as \cref{thm:Phi_normalisable} need not hold;
the following result considers only measures which \emph{can} be normalised.

Observe also that the hypotheses of \cref{thm:main} always imply $E$-coercivity of $I^{y}$:
using the compact embedding \eqref{eq:compact_embedding} of $E$ in $X$ and the lower bound \eqref{eq:Phi_lower_bound} gives
\begin{equation*}
	I^{y}(u) \defeq \Phi(u) + \frac{1}{2} \norm{u}_{E}^2 \geq K(\eta) + \left(\frac{1}{2} - C^2\eta\right)\norm{u}_{E}^2 \text{~~for all $u \in E$,}
\end{equation*}
and selecting $\eta > 0$ sufficiently small ensures that $\tfrac{1}{2} - C^2 \eta > 0$.

It is not clear whether coercivity is sufficient to obtain a strong mode, and this question is left to future work.

\begin{proposition}[Weak MAP estimators for Bayesian posteriors with Gaussian priors]
	\label{prop:weak_modes}
	Let $X$ be a separable Banach space and let $\mu_0$ be a centred nondegenerate Gaussian measure.
	Suppose that $\mu^y$ is a \emph{probability measure} of the form \eqref{eq:Bayesian_posterior} for some continuous potential $\Phi \colon X \to \Reals$.
	Suppose also that the posterior OM functional $I^y(u) \defeq \Phi(u) + \frac{1}{2}\norm{u}_E^2$ is \emph{$E$-coercive}, i.e.\ there exists $A \in \Reals$ and $c > 0$ such that 
	\begin{equation*}
		A + I^y(u) \geq c\norm{u}_{E}^{2} \text{~~for all $u \in E$},
	\end{equation*}
	or equivalently, using the definition of $I^{y}$,
	\begin{equation*}
		A + \Phi(u) > -\frac{1}{2} \norm{u}_{E}^2 \text{~~for all $u \in E$.}
	\end{equation*}
	Then $\mu^y$ has a weak mode.
\end{proposition}

\begin{proof}
	The prior $\mu_0$ has OM functional $I_0(u) = \frac{1}{2}\norm{u}_{E}^{2}$ as described in \eqref{eq:Gaussian_OM} and \cref{cor:Gaussian_measure_M-prop} proves that property $M(\mu_0, E)$ holds.
	Hence, by \cref{prop:OM_minimisers_are_modes}, the posterior has OM functional $I^y(u) = \frac{1}{2} \norm{u}_{E}^{2} + \Phi(u)$ and property $M(\mu^y, E)$ holds, and furthermore weak modes coincide with minimisers of $I^y$.
	It remains to show that $I^y$ does have a minimiser.

	First, note that $\Phi$ is $E$-weakly continuous: 
	if $u_n \rightharpoonup u$ weakly in $E$, then by the compact embedding $u_n \to u$ strongly in $X$ and thus $\Phi(u_n) \to \Phi(u)$ by strong continuity of $\Phi$ in $X$.
	As the $E$-norm is also clearly weakly lower semicontinuous, the OM functional $I^y$ must be $E$-weakly lower semicontinuous.
	As $I^y$ is also coercive, it has a minimiser in $E$ by the direct method of the calculus of variations:
	take a sequence $(h_n)_{n \in \Naturals} \subset E$ with $I^y(h_n) < \inf_{u \in E} I^y(u) + \tfrac{1}{n}$, and observe that it is $E$-bounded by coercivity;
	passing to an $E$-weakly convergent subsequence with limit $h^{\star}$ and using the weak lower semicontinuity of $I^y$ proves that $h^{\star}$ is a minimiser of $I^y$.
	This minimiser is a weak mode by \cref{prop:OM_minimisers_are_modes}.
\end{proof}

\subsection{Strong MAP estimators}
We now prove the main theorem on the existence of strong MAP estimators.
The strategy of the proof is similar in spirit to the prior work of \citet{DashtiLawStuartVoss2013}, \citet{Kretschmann2019,Kretschmann2022} and \citet{KlebanovWacker2022}.

In the proof of \citet{DashtiLawStuartVoss2013}, the explicit Anderson inequality \eqref{eq:explicit_Anderson_inequality_in_X} is first used to show that any family $(x_{r}^{\star})_{r > 0}$ of maximisers of the posterior radius-$r$ ball mass $x \mapsto \mu^{y}(\cball{x}{r})$ must be bounded in $X$ under some regularity assumptions on $\Phi$. 
Next, a weakly convergent subsequence is extracted, and Lemma~3.7 and Lemma~3.9 of \citet{DashtiLawStuartVoss2013} can be used to show that if the limit is not in $E$ or the convergence is not strong, then
\begin{equation*}
	\frac{\gamma(\cball{x_{r}^{\star}}{r})}{\gamma(\cball{0}{r})} \to 0 \text{~~~~as $r \to 0$.}
\end{equation*}
This yields a contradiction because the assumptions on $\Phi$ mean this ratio cannot converge to zero, showing that the limit point lies in $E$ and convergence is strong in $X$.
Finally, this limit point is shown to be both a strong MAP estimator and an OM minimiser.

\citet{KlebanovWacker2022} point out that it is not obvious that the radius-$r$ maximisers exist and show that the proof can be adapted to use a family $(x_{r})_{r > 0}$ of ``approximate maximisers'' nearly attaining the supremal radius-$r$ mass instead.
\citet{KlebanovWacker2022} call such a family an \defterm{asymptotic maximising family}.

\begin{definition}
	\label{def:AMF}

	Let $X$ be a metric space.
	An \defterm{asymptotic maximising family} (AMF) for $\mu \in \prob{X}$ is a net $(x_r)_{r > 0} \subset X$ such that, for some increasing function $\varepsilon \colon [0, \infty) \to [0, 1)$ with $\lim_{r \to 0} \varepsilon(r) = 0$,
	\begin{equation*}
		\crcdf{x_r}{r} \geq \bigl(1 - \varepsilon(r)\bigr) M_r.
	\end{equation*}
\end{definition}

Every measure has at least one AMF, though in general there may not exist any point $x_{r}^{\star} \in X$ such that $\crcdf{x_r^{\star}}{r} = M_r$.
If $X$ is a Hilbert space, then the radius-$r$ maximisers $x_{r}^{\star}$ do always exist \citep[Corollary~A.9]{LambleySullivan2023}, but we will use AMFs to avoid further discussion about these maximisers.
The next result summarises the connection between AMFs and small-ball modes, which is explored in greater detail by \citet[Theorem~4.11]{LambleySullivan2023}.

\begin{proposition}
	\label{prop:generalised_strong_AMF}
	Let $X$ be a separable Banach space and suppose that $\mu \in \prob{X}$.
	\begin{enumerate}[label=(\alph*)]
		\item Suppose that $x^{\star}$ is a generalised strong mode for $\mu$.
		Then $x^\star$ is a limit point of some AMF $(x_{r})_{r > 0} \subset X$.

		\item Suppose that $(x_{r})_{r > 0} \subset X$ is an AMF for $\mu$ which converges to $x^{\star}$ along every subsequence.
		Then $x^{\star}$ is a generalised strong mode.
	\end{enumerate}
\end{proposition}

\begin{proof}
	Pick any sequence $(r_{n})_{n \in \Naturals} \to 0$ and choose a corresponding sequence $(x_{r_n})_{n \in \Naturals} \to x^{\star}$ from the definition of a generalised strong mode (\Cref{def:generalised_mode}).
	Selecting any AMF $(x_{r})_{r > 0}$ with this subsequence $(x_{r_n})_{n \in \Naturals}$ proves the first claim.
	For the second claim, let $\varepsilon \colon [0, \infty) \to [0, 1)$ denote the function corresponding to the AMF $(x_{r})_{r > 0}$;
	for any sequence $(r_{n})_{n \in \Naturals} \to 0$, it follows by definition that
	\begin{equation*}
		\lim_{n \to \infty} \frac{\crcdf{x_{r_n}}{r_n}}{M_{r_n}} = \lim_{n \to \infty} 1 - \varepsilon(r_n) = 1,
	\end{equation*}
	proving that $x^\star$ is a generalised strong mode.
\end{proof}

Aside from the issues associated with radius-$r$ maximisers, the proof of \citet{DashtiLawStuartVoss2013} omits some technical results which were later proved by \citet{Kretschmann2019}.
\citet{KlebanovWacker2022} argue that the proof also relies on several properties that do not hold in an arbitrary separable Banach space $X$.
To give just one example, the step passing from a bounded sequence to a weakly convergent subsequence requires additional hypotheses, e.g.\ reflexivity of $X$.

To resolve this, \citet{KlebanovWacker2022} first establish the proof when $X$ is a separable Hilbert space. 
In this setting, any Gaussian measure $\gamma$ is characterised by its mean $m \in X$ and covariance operator $\mathcal{C} \colon X \to X$, so by working in an eigenbasis of $\mathcal{C}$, one can reduce to the case $X = \ell^{2}(\Naturals; \Reals)$ with $\mu_{0} = \bigotimes_{n \in \Naturals} N(0, \sigma_{n}^{2})$, with the Cameron--Martin norm given by a simple reweighting of the $\ell^{2}$-norm.
\citet{KlebanovWacker2022} then extend to the case $X = \ell^{p}(\Naturals; \Reals)$, $1 \leq p < \infty$, with $\mu_{0} = \bigotimes_{n \in \Naturals} N(0, \sigma_{n}^{2})$;
unlike in the Hilbert case, not all Gaussian measures on $\ell^{p}$ can be expressed in this product form.
Even this generalisation is nontrivial since the Cameron--Martin norm can no longer be expressed as a reweighting of the $X$-norm.
This motivates a technical convexification argument to bridge the gap between the two norms, making use of the diagonal structure of the prior to write the $E$-norm in terms of the canonical sequence-space basis.
It is challenging to generalise this approach further given the heavy dependence on the diagonal structure.

We overcome this difficulty by using the explicit Anderson inequality of \cref{prop:CM_bound}.
As discused in \cref{sec:small-ball_Gaussian_Banach}, this is more natural than the bound \eqref{eq:explicit_Anderson_inequality_in_X} used in prior work because the behaviour of $\gamma$ is fully determined by its Cameron--Martin space, and the Cameron--Martin space has more favourable topological properties.
\cref{prop:CM_bound} first allows us to show that any AMF is bounded in $X$, and \cref{cor:Gaussian_measure_M-prop} shows that any AMF is closely approximated in $X$ by a sequence bounded \emph{in $E$}.
This sequence has an $E$-weakly convergent subsequence regardless of the choice of $X$, and applying the compact embedding of $E$ in $X$ yields strong convergence of this subsequence in $X$.

This approach avoids the need to explicitly prove Lemma~3.7 and Lemma~3.9 of \citet{DashtiLawStuartVoss2013}, since the necessary claims can be derived directly from \cref{prop:CM_bound} and \cref{cor:Gaussian_measure_M-prop}. 

We will later show in \cref{lem:AMF_limit_point_is_strong} that a limit point of an AMF is a strong mode for the Bayesian posterior $\mu^{y}$;
combining this result with the existence of limit points proven in the following result completes the proof of \cref{thm:main}.

\begin{lemma}[Limit points of AMFs for Bayesian posteriors]
	\label{lem:AMF_has_limit_point}
	Under the assumptions of \cref{thm:main}, if $(x_r)_{r > 0} \subset X$ is an AMF for $\mu^y$, then:
	\begin{enumerate}[label=(\alph*)]
		\item any limit point of $(x_r)_{r > 0}$ lies in $E$;
		\item the net $(x_r)_{r > 0}$ has at least one limit point.
	\end{enumerate}
\end{lemma}

\begin{proof}
	Fix any decreasing sequence $(r_n)_{n \in \Naturals} \to 0$.
	By \cref{prop:CM_bound} and the compact embedding \eqref{eq:compact_embedding}, we have
	\begin{equation}\label{eq:AMF_bounded_upper_bound}
		\frac{\mu_0(\cball{x_{r_n}}{r_n})}{\mu_0(\cball{0}{r_n})} \leq \exp\left(-\frac{1}{2C^{2}} \min_{h \in E \cap \cball{x_{r_n}}{r_n}} \|h\|_X^{2}\right) \leq \exp\left(-\frac{1}{2C^{2}} \inf_{x \in \cball{x_{r_n}}{r_{n}}} \norm{x}_{X}^{2} \right).
	\end{equation}
	On the other hand, let $\varepsilon$ be the function corresponding to the AMF $(x_{r})_{r > 0}$;
	using the lower bound \eqref{eq:Phi_lower_bound} on $\Phi$ and picking $\delta > 0$ from the definition of continuity such that $\absval{\Phi(x) - \Phi(0)} < 1$ for $\absval{x} < \delta$, we see that for all $n$ such that $r_n < \delta$, the following lower bound holds:
	\begin{align*}
		\exp\left(-K(\eta) + \eta \sup_{x \in \cball{x_{r_n}}{r_{n}}} \norm{x}_{X}^2\right) \mu_0(\cball{x_{r_n}}{r_n})
		&\geq \mu^y(\cball{x_{r_n}}{r_n}) \\
		&\geq \bigl(1 - \varepsilon(r_n)\bigr) M_{r_n} \\
		&\geq \bigl(1-\varepsilon(r_n)\bigr) \int_{\cball{0}{r_n}} \exp\bigl(-\Phi(x)\bigr)\,\mu_0(\rd x) \\
		&\geq \bigl(1 - \varepsilon(r_n)\bigr) \exp\bigl(-\Phi(0) - 1\bigr) \mu_0(\cball{0}{r_n}).
	\end{align*}
	This inequality gives
	\begin{equation} \label{eq:AMF_bounded_lower_bound}
		\frac{\mu_0(\cball{x_{r_n}}{r_n})}{\mu_0(\cball{0}{r_n})} \geq \bigl(1 - \varepsilon(r_n)\bigr) \exp\left(K(\eta) - \Phi(0) - 1 -\eta \sup_{x \in \cball{x_{r_n}}{r_{n}}} \norm{x}_{X}^{2} \right),
	\end{equation}
	and combining this bound with \eqref{eq:AMF_bounded_upper_bound} yields
	\begin{equation*}
		 \bigl(1-\varepsilon(r_n)\bigr)\exp\bigl(K(\eta) - \Phi(0) - 1\bigr) \leq \exp\left(\eta \sup_{x \in \cball{x_{r_n}}{r_{n}}} \norm{x}_{X}^{2} - \frac{1}{2C^{2}} \inf_{x \in \cball{x_{r_n}}{r_{n}}} \norm{x}_{X}^2 \right).
	\end{equation*}
	This implies that the sequence $(x_{r_n})_{n \in \Naturals}$ is bounded:
	if it were not, then setting $\eta < \tfrac{1}{2C^{2}}$ would give the contradiction
	\begin{align*}
		0 &< \exp\bigl(K(\eta) - \Phi(0) - 1\bigr) \\
	      &\leq \liminf_{n \to \infty} \exp\left(\eta \sup_{x \in \cball{x_{r_n}}{r_{n}}} \norm{x}_{X}^{2} - \frac{1}{2C^{2}} \inf_{x \in \cball{x_{r_n}}{r_{n}}} \norm{x}_{X}^2 \right) \\
		  &\leq \liminf_{n \to \infty} \exp\left(\eta \left(\inf_{x \in \cball{x_{r_n}}{r_{n}}} \norm{x}_{X} + 2r_{n}\right)^{2} - \frac{1}{2C^{2}} \inf_{x \in \cball{x_{r_n}}{r_{n}}} \norm{x}_{X}^2 \right) = 0.
	\end{align*}
	As $(x_{r_{n}})_{n \in \Naturals}$ is bounded, \eqref{eq:AMF_bounded_lower_bound} implies that there is a constant $L > 0$ such that 
	\begin{equation*}
		0 < L \bigl(1 - \varepsilon(r_n)\bigr) \leq \frac{\mu_0(\cball{x_{r_{n}}}{r_{n}})}{\mu_0(\cball{0}{r_{n}})}.
	\end{equation*}
	Thus, again using the upper bound provided by \cref{prop:CM_bound}, we see that 
	\begin{equation*}
		0 < \liminf_{n \to \infty} L \bigl(1 - \varepsilon(r_n)\bigr) \leq \liminf_{n \to \infty} \frac{\mu_0(\cball{x_{r_{n}}}{r_{n}})}{\mu_0(\cball{0}{r_{n}})} \leq \liminf_{n \to \infty} \exp\left(-\min_{h \in E \cap \cball{x_{r_{n}}}{r_{n}}}\frac{1}{2} \|h\|_{E}^{2}\right).
	\end{equation*}
	It then follows from \cref{cor:Gaussian_measure_M-prop} that $(x_{r_n})_{n \in \Naturals}$ has a further subsequence converging to some point $x^{\star} \in E$.
	In particular, if $(x_{r_n})_{n \in \Naturals}$ is a convergent sequence, then the limit must lie in $E$.
\end{proof}

As discussed, \citet[Theorem~2.8]{KlebanovWacker2022} proved that a limit point of an AMF for the posterior is a strong mode in the sequence-space setting. 
The following lemma generalises this result to any separable Banach space and slightly weakens the hypotheses required on the potential to be merely continuous rather than locally Lipschitz.

\begin{lemma}[Limit points of AMFs are strong modes]
	\label{lem:AMF_limit_point_is_strong}
	Under the assumptions of \cref{thm:main}, any $X$-strong limit point $x^{\star} \in E$ of an AMF $(x_r)_{r > 0} \subset X$ for $\mu^y$ is a strong mode.
\end{lemma}

\begin{proof}
	Let $x^{\star} \in E$ be some limit point of $(x_{r})_{r > 0}$. 
	To show that $x^{\star}$ is a strong mode, it suffices to check that for any $(r_{n})_{n \in \Naturals} \to 0$,
	\begin{equation} \label{eq:strong_mode_along_subsequences}
		\lim_{n \to \infty} \frac{\mu^{y}(\cball{x^{\star}}{r_{n}})}{M_{r_{n}}} \geq 1.
	\end{equation}
	Indeed, it would be enough to show that any $(r_{n})_{n \in \Naturals} \to 0$ has a further subsequence such that \eqref{eq:strong_mode_along_subsequences} holds along that subsequence:
	this follows from the fact that if $(u_n)_{n \in \Naturals}$ is an arbitrary real sequence and any subsequence of $(u_n)_{n \in \Naturals}$ has a further subsequence converging to $u$, then $u_n \to u$.

	Hence, take any sequence $(r_{n})_{n \in \Naturals} \to 0$; by \cref{lem:AMF_has_limit_point} we may pass to a subsequence of $(r_{n})_{n \in \Naturals}$, which will not be relabelled, such that $(x_{r_n})_{n \in \Naturals}$ converges to some $y^{\star} \in E$.
	As $\Phi$ is continuous, for any $\varepsilon > 0$ there exists $\delta > 0$ such that for any $x \in \cball{y^\star}{\delta}$, it follows that $\absval{\Phi(y^{\star}) - \Phi(x)} < \varepsilon$.
	Since $(r_n)_{n \in \Naturals} \to 0$ and $(x_{r_n})_{n \in \Naturals} \to y^{\star}$, there exists $N \in \Naturals$ such that $\absval{r_n} < \tfrac{\delta}{2}$ and $\norm{x_{r_n} - y^{\star}}_X < \tfrac{\delta}{2}$ for $n \geq N$.
	Hence for such $n$ and any $x \in \cball{x_{r_n}}{r_n}$, we have $\Phi(x_{r_n}) - \Phi(x) < 2\varepsilon$.
	This implies that
	\begin{align}
		\notag
		\frac{\mu^y(\cball{x_{r_n}}{r_n})}{\mu^y(\cball{y^{\star}}{r_n})} &= \exp\bigl(\Phi(y^{\star}) - \Phi(x_{r_n})\bigr) \frac{\int_{\cball{x_{r_n}}{r_n}} \exp\bigl(\Phi(x_{r_n}) - \Phi(x)\bigr)\,\mu_0(\rd x)}{\int_{\cball{y^{\star}}{r_n}} \exp\bigl(\Phi(y^{\star}) - \Phi(x)\bigr)\,\mu_0(\rd x)} \\
		\label{eq:posterior_mass_ratio}
		&< \exp(4\varepsilon) \frac{\mu_0(\cball{x_{r_n}}{r_n})}{\mu_0(\cball{y^{\star}}{r_n})}.
	\end{align}
	By \cref{cor:Gaussian_measure_limits}, we may pass to a further subsequence of $(x_{r_n})_{n \in \Naturals}$ without relabelling such that
	\begin{equation*}
		\limsup_{n \to \infty} \frac{\mu_0(\cball{x_{r_n}}{r_n})}{\mu_0(\cball{y^{\star}}{r_n})} \leq 1,
	\end{equation*}
	and thus taking the $\limsup$ in \eqref{eq:posterior_mass_ratio} and infimising over $\varepsilon > 0$ yields
	\begin{equation*}
		\left(\liminf_{n \to \infty} \frac{\mu^y(\cball{y^{\star}}{r_n})}{\mu^y(\cball{x_{r_n}}{r_n})}\right)^{-1} = \limsup_{n \to \infty} \frac{\mu^y(\cball{x_{r_{n}}}{r_{n}})}{\mu^y(\cball{y^{\star}}{r_{n}})} \leq \limsup_{n \to \infty} \frac{\mu_0(\cball{x_{r_{n}}}{r_{n}})}{\mu_0(\cball{y^{\star}}{r_{n}})} \leq 1.
	\end{equation*}
	Since $(x_{r_{n}})_{n \in \Naturals}$ is a subsequence of an AMF, the previous equation implies that
	\begin{equation*}
		1 \leq \lim_{n \to \infty} \frac{\mu^y(\cball{x_{r_{n}}}{r_{n}})}{M_{r_{n}}} \liminf_{n \to \infty} \frac{\mu^y(\cball{y^{\star}}{r_{n}})}{\mu^y(\cball{x_{r_{n}}}{r_{n}})} = \lim_{n \to \infty} \frac{\mu^y(\cball{y^{\star}}{r_{n}})}{M_{r_{n}}}  \leq 1.
	\end{equation*}
	In particular, the point $x^{\star}$ fixed at the start of the proof is a limit point of the AMF $(x_r)_{r > 0}$, i.e.\ there exists $(s_n)_{n \in \Naturals} \to 0$ such that $(x_{s_n})_{n \in \Naturals} \to x^{\star}$, so the above argument implies the existence of a subsequence such that
	\begin{equation} \label{eq:x_star_maximises_along_sn}
		\lim_{n \to \infty} \frac{\mu^y(\cball{x^{\star}}{s_n})}{M_{s_n}} = 1.
	\end{equation}
	This does not yet hold for \emph{every} sequence $(r_n)_{n \in \Naturals} \to 0$, only the specific sequence $(s_n)_{n \in \Naturals}$.
	To complete the proof, fix an arbitrary $(r_n)_{n \in \Naturals} \to 0$ and $y^{\star} \in E$ as above.
	As $\mu^y$ has an OM functional $I^y$ defined on $E$,
	\begin{equation} \label{eq:OM_functional_rn_sn}
		\lim_{n \to \infty} \frac{\mu^{y}(\cball{x^{\star}}{r_n})}{\mu^{y}(\cball{y^{\star}}{r_n})} = 
		\lim_{n \to \infty} \frac{\mu^{y}(\cball{x^{\star}}{s_n})}{\mu^{y}(\cball{y^{\star}}{s_n})} = 
		\lim_{r \to 0} \frac{\mu^y(\cball{x^{\star}}{r})}{\mu^y(\cball{y^{\star}}{r})} = I^y(y^{\star}) - I^y(x^{\star}).
	\end{equation}
	Hence, using \eqref{eq:x_star_maximises_along_sn} and \eqref{eq:OM_functional_rn_sn}, it follows that
	\begin{align*}
		\lim_{n \to \infty} \frac{\mu^{y}(\cball{x^{\star}}{r_n})}{M_{r_n}} 
		&= \lim_{n \to \infty} \frac{\mu^{y}(\cball{y^{\star}}{r_n})}{M_{r_n}} \lim_{n \to \infty} \frac{\mu^{y}(\cball{x^{\star}}{r_n})}{\mu^{y}(\cball{y^{\star}}{r_n})} \\
		&= \lim_{n \to \infty} \frac{\mu^{y}(\cball{x^{\star}}{s_n})}{\mu^{y}(\cball{y^{\star}}{s_n})}
		\geq \lim_{n \to \infty} \frac{\mu^{y}(\cball{x^{\star}}{s_n})}{M_{s_n}} = 1.
	\end{align*}
	Hence, for any $(r_n)_{n \in \Naturals} \to 0$, there is a further subsequence for which \eqref{eq:strong_mode_along_subsequences} holds, and thus $x^{\star}$ is a strong mode.
\end{proof}

\begin{proof}[Proof of \cref{thm:main}]
	\begin{enumerate}[label=(\alph*)]
		\item
		By \cref{lem:AMF_has_limit_point}, any AMF $(x_r)_{r > 0}$ has an $X$-strong limit point $x^{\star} \in E$, and this limit point is a strong mode by \cref{lem:AMF_limit_point_is_strong}.

		\item
		\Cref{lem:all_weak_or_strong} proves that strong and weak modes coincide as a strong mode exists, and \cref{prop:OM_minimisers_are_modes} shows that weak modes coincide with minimisers of the OM functional.
		As any generalised strong mode must be the limit point of an AMF (\Cref{prop:generalised_strong_AMF}) and \cref{lem:AMF_has_limit_point} implies that such a point lies in $E$, \cref{lem:AMF_limit_point_is_strong} implies that the generalised strong mode is also a strong mode.
		\qedhere
	\end{enumerate}
\end{proof}

\section{Consistency of MAP estimators}
\label{sec:consistency}

We return to the additive-noise Bayesian inverse problem \eqref{eq:additive-noise_BIP} discussed in the introduction.
Suppose that $X$ is a separable Banach space, $Y \defeq \Reals^{d}$ and $\mathcal{G} \colon X \to Y$.
For simplicity, we restrict attention to the case of mean-zero Gaussian noise $\xi$ and Gaussian prior $\mu_{0}$, giving the model
\begin{equation} \label{eq:additive-noise_BIP_Gaussian_noise}
	y = \mathcal{G}(x) + \xi,~~~\xi \sim N(0, \Sigma),~~~x \sim \mu_{0}.
\end{equation}
Under the frequentist assumption that there is a fixed true parameter $x^{\dagger} \in X$, consistency theory studies the behaviour of the posterior and point estimators --- which depend on the random observations $y_{1}, \dots, y_{N}$ --- in the infinite-data or small-noise limit.

Classically, a sequence of posterior measures is consistent at $x^{\dagger}$ if, for any neighbourhood $U$ of $x^{\dagger}$, the posterior measure of $U^{C}$ converges to zero in probability \citep{GhosalVanDerVaart2017}.
For Bayesian inverse problems, this notion is often too restrictive:
the parameter $x^{\dagger}$ need not even be identifiable from the model, because there may exist $x \neq x^{\dagger}$ such that $\mathcal{G}(x) = \mathcal{G}(x^{\dagger})$.
Indeed, if $\mathcal{G}$ is a bounded linear operator, then $\mathcal{G}(x^{\dagger}) = \mathcal{G}(x^{\dagger} + z)$ for any $z \in \ker \mathcal{G} \neq \emptyset$, so $x^{\dagger}$ is never identifiable, and thus one cannot expect posterior consistency to hold.
Posterior consistency is often possible to show if $\mathcal{G}$ is known more explicitly, e.g.\ if it is the solution operator for a partial differential equation \citep[see][]{KnapikVanDerVaartVanZanten2011, AgapiouLarssonStuart2013, Vollmer2013}, but we focus on the general case of a possibly nonlinear operator $\mathcal{G} \colon X \to Y$.

In a similar vein, one cannot expect a sequence of MAP estimators to be consistent estimators of $x^{\dagger}$, i.e.\ the MAP estimators need not converge in probability to $x^{\dagger}$.
It is instead typical to study a weaker notion of consistency for MAP estimators, where one identifies a limit point $x^{\star}$ of any sequence of MAP estimators and shows that $\mathcal{G}(x^{\dagger}) = \mathcal{G}(x^{\star})$ \citep{DashtiLawStuartVoss2013,Dunlop2019,AgapiouBurgerDashtiHelin2018}.
The results of \citet[Section~4]{DashtiLawStuartVoss2013} on the consistency of MAP estimators in the setting of \eqref{eq:additive-noise_BIP_Gaussian_noise} depend on the correspondence between strong MAP estimators and OM minimisers, and on the existence of strong MAP estimators.
Thus, \cref{thm:main} can be used to extend the applicability of these consistency results from separable Hilbert spaces to arbitrary separable Banach spaces.

\paragraph{Infinite-data limit.}
Assume that data $y_{1}, \dots, y_{N}, \dots$ are repeated observations from the model \eqref{eq:additive-noise_BIP_Gaussian_noise} assuming the fixed parameter $x^{\dagger}$, i.e.\ 
\begin{equation} \label{eq:observation_model_infinite-data_limit}
	y_{n} = \mathcal{G}(x^{\dagger}) + \xi_{n},~~~\xi_{n} \overset{\text{i.i.d.}}{\sim} N(0, \Sigma).
\end{equation}
Using Bayes' rule (\cref{thm:Bayes}) with the concatenated observation vector $\bm{y} = (y_{1}, \dots, y_{N}) \in \Reals^{Nd}$ and the appropriate product noise distribution, we see that the posterior  is absolutely continuous with respect to $\mu_{0}$ and has density
\begin{equation} \label{eq:posterior_Gaussian_noise}
	\frac{\rd \mu^{y_{1}, \dots, y_{N}}}{\rd \mu_{0}}(x) \propto \exp\bigl(-\Phi^{y_{1}, \dots, y_{N}}(x)\bigr),~~~~\Phi^{y_{1}, \dots, y_{N}}(x) = \frac{1}{2} \sum_{n = 1}^{N} \Norm{\mathcal{G}(x) - y_{n}}_{\Sigma}^{2},
\end{equation}
with $\Norm{\quark}_{\Sigma} \defeq \norm{\Sigma^{-1/2} \quark}$, analogous to the potential derived in \eqref{eq:Gaussian-noise_BIP_potential}.
As discussed in \cref{sec:BIPs}, the normalisation constant for the density can be absorbed into $\Phi^{y_{1}, \dots, y_{N}}$ by adding a constant depending on $y_{1}, \dots, y_{N}$ to the potential.
Assuming that $\mathcal{G}$ is continuous, the conclusions of \cref{thm:main} hold because $\Phi^{y_{1}, \dots, y_{N}}$ is continuous and bounded below, so for each $N$ there exists at least one strong MAP estimator $x_{N}$ of $\mu^{y_{1}, \dots, y_{N}}$, which can be obtained by minimising the OM functional
\begin{equation*}
	I^{y_{1}, \dots, y_{N}}(x) = \frac{1}{2} \norm{x}_{E}^{2} + \Phi^{y_{1}, \dots, y_{N}}(x)
\end{equation*}
over all $x \in E$.
Following the proof of \citet[Theorem~4.1 and Corollary~4.3]{DashtiLawStuartVoss2013}, we obtain a weak consistency result for MAP estimators which applies more generally in any separable Banach space with any continuous observation operator $\mathcal{G}$.

\begin{theorem}
	\label{thm:consistency_infinite-data_limit}

	Let $X$ be a separable Banach space and let $\mathcal{G} \colon X \to \Reals^{d}$ be continuous. 
	Let $x^{\dagger} \in X$ be arbitrary.
	Suppose that the data $y_{1}, \dots, y_{N}, \dots$ are generated from the observation model \eqref{eq:observation_model_infinite-data_limit} and $\mu^{y_{1}, \dots, y_{N}}$ is the corresponding posterior. 
	For each $N \in \Naturals$, let $x_{N}$ be any strong MAP estimator of $\mu^{y_{1}, \dots, y_{N}}$.
	Then:
	\begin{enumerate}[label=(\alph*)]
		\item there is a subsequence of $(\mathcal{G}(x_{N}))_{N \in \Naturals}$ converging to $\mathcal{G}(x^{\dagger})$ almost surely; 
		\item if $x^{\dagger} \in E$, there is a subsequence of $(x_{N})_{N \in \Naturals}$ converging to some $x^{\star} \in E$ weakly in $E$ almost surely, and $\mathcal{G}(x^{\star}) = \mathcal{G}(x^{\dagger})$.
	\end{enumerate}
\end{theorem}

Since the proof is a straightforward adaptation of that of \citet{DashtiLawStuartVoss2013}, it is omitted.

\paragraph{Small-noise limit.}
Assume that $y_{1}, \dots, y_{N}, \dots$ are observations from the sequence of models
\begin{equation} \label{eq:observation_model_small-noise_limit}
	y_{n} = \mathcal{G}(x^{\dagger}) + \frac{1}{n} \xi_{n},~~~\xi_{n} \overset{\text{i.i.d.}}{\sim} N(0, \Sigma).
\end{equation}
Unlike in the infinite-data scenario, we consider the sequence of posteriors $\mu^{y_{N}}$ obtained from just a single observation;
again applying Bayes' rule we obtain that $\mu^{y_{N}}$ is absolutely continuous with respect to $\mu_{0}$ and has density
\begin{equation*}
	\frac{\rd \mu^{y_{N}}}{\rd \mu_{0}}(x) \propto \exp\bigl(-\Phi^{y_{N}}(x)\bigr),~~~~\Phi^{y_{N}}(x) = \frac{N^{2}}{2} \norm{\mathcal{G}(x) - y_{N}}_{\Sigma}^{2}.
\end{equation*}
As before, the potential satisfies the hypotheses of \cref{thm:main}, and so at least one strong MAP estimator exists for each posterior $\mu^{y_{N}}$.
Following the proof of \citet[Theorem~4.4 and Corollary~4.5]{DashtiLawStuartVoss2013}, we also obtain weak consistency for MAP estimators in the small-noise setting.

\begin{theorem}
	\label{thm:consistency_small-noise_limit}

	Let $X$ be a separable Banach space and let $\mathcal{G} \colon X \to \Reals^{d}$ be continuous. 
	Let $x^{\dagger} \in X$ be arbitrary.
	Suppose that the data $y_{1}, \dots, y_{N}, \dots$ are generated from the observation model \eqref{eq:observation_model_small-noise_limit} and $\mu^{y_{N}}$ is the corresponding posterior. 
	For each $N \in \Naturals$, let $x_{N}$ be any strong MAP estimator of $\mu^{y_{N}}$.
	Then:
	\begin{enumerate}[label=(\alph*)]
		\item there is a subsequence of $(\mathcal{G}(x_{N}))_{N \in \Naturals}$ converging to $\mathcal{G}(x^{\dagger})$ almost surely;
		\item if $x^{\dagger} \in E$, there is a subsequence of $(x_{N})_{N \in \Naturals}$ converging to some $x^{\star} \in E$ weakly in $E$ almost surely, and $\mathcal{G}(x^{\star}) = \mathcal{G}(x^{\dagger})$.
	\end{enumerate}
\end{theorem}

\section{Closing remarks}
\label{sec:conclusion}

MAP estimators provide a simple summary of the posterior distribution, but in the nonparametric setting it is not straightforward even to verify that MAP estimators exist.
This article has shown that Bayesian inverse problems defined on any separable Banach space with a Gaussian prior have well-defined strong MAP estimators under very mild conditions on the forward problem.
The fact that MAP estimators correspond with minimisers of a Tikhonov functional is an important justification for the Bayesian approach, and this article has also extended the connection between MAP estimators and variational minimisers to the Banach setting.
As a corollary of \cref{thm:main} on the existence of strong MAP estimators, this article also extends results on the consistency of MAP estimators to additive-noise Bayesian inverse problems set in any separable Banach space.

The strategy adopted here depends on two essential points:
the statistical structure of the Bayesian inverse problem \eqref{eq:additive-noise_BIP}, which ensures the posterior is absolutely continuous with respect to the prior, and the topological structure provided by the Gaussian prior through the compactly embedded Cameron--Martin space.

In more general settings, such as those where the observed quantity has infinite dimension, it need not be the case that the posterior is absolutely continuous with respect to the prior \citep[Remark~3.8]{Stuart2010}.
While the small-ball theory for modes (\Cref{subsection:modes}) does not depend on this absolute continuity, new techniques are needed to translate statements from prior to posterior without a density to relate the two.

Though this article has restricted attention to the case that $X$ is a separable Banach space, the results on Gaussian measures used in this article hold more generally for Radon Gaussian measures on a locally convex space $X$, and this would form a natural extension of this work.

Another possible extension is to other priors with similar structure used in nonparametric Bayesian inverse problems, such as the $p$-exponential priors of \citet{AgapiouDashtiHelin2021}.
As pointed out by \citet[Remark~2.14]{AgapiouDashtiHelin2021}, a bound analogous to \cref{prop:CM_bound} is more challenging for non-Gaussian $p$-exponential measures (i.e.\ $p < 2$) because the appropriate analogue of the Cameron--Martin space is not a Hilbert space;
on the other hand, $p$-exponential measures are defined on subspaces of the countable product space $\Reals^{\infty}$, which provides a useful topological structure not present in an arbitrary Banach space.

It would also be interesting to know whether the hypothesis of coercivity in \cref{prop:weak_modes} --- which was sufficient to prove the existence of weak modes --- would also suffice for proving the existence of strong modes.

\section*{Acknowledgements}
\addcontentsline{toc}{section}{Acknowledgements}

The author thanks Ilja Klebanov and Tim Sullivan for helpful feedback and comments.%

The author is supported by the Warwick Mathematics Institute Centre for Doctoral Training and gratefully acknowledges funding from the University of Warwick and the UK Engineering and Physical Sciences Research Council (Grant number: EP/W524645/1).
For the purpose of open access, the author has applied a Creative Commons Attribution (CC BY) licence to any Author Accepted Manuscript version arising.

\section*{Data availability statement}
\addcontentsline{toc}{section}{Data availability statement}
 No new data were created or analysed in this study.

\bibliographystyle{abbrvnat}
\bibliography{references}
\addcontentsline{toc}{section}{References}

\end{document}